\documentclass[final]{myarticle}

\usepackage[T1]{fontenc}
\usepackage{dsfont}
\usepackage{mathrsfs}
\usepackage[colorlinks, citecolor=blue, linkcolor=blue]{hyperref}
\usepackage{xcolor}
\usepackage{mathscinet}
\usepackage{latexsym}
\usepackage{amsthm}
\usepackage{amssymb}
\usepackage{amsfonts}
\usepackage{amsmath}
\usepackage{longtable}
\usepackage{graphicx}
\usepackage{multirow}
\usepackage{multicol}
\usepackage{amsfonts, amsmath}
\usepackage{latexsym,bm,amsfonts,amssymb,pifont,mathbbol,bbm}
\usepackage{extarrows}
\usepackage{verbatim}
\allowdisplaybreaks
\usepackage[a4paper]{geometry}

\DeclareMathAlphabet{\mathpzc}{OT1}{pzc}{m}{it} 

\newtheorem{theorem}{Theorem}[section]
\newtheorem{thm}[theorem]{Theorem}
\newtheorem{lemma}[theorem]{Lemma}
\newtheorem{lem}[theorem]{Lemma}

\newtheorem{proposition}[theorem]{Proposition}
\newtheorem{prop}[theorem]{Proposition}
\newtheorem{corollary}[theorem]{Corollary}

\theoremstyle{definition}

\theoremstyle{remark}
\newtheorem{remark}[theorem]{Remark}

\numberwithin{equation}{section}

\newcommand{\set}[1]{\left\{#1\right\}}
\newcommand{\im}{\mathrm{i}}
\newcommand{\mi}{\mathrm{i}}
\newcommand{\abs}[1]{\left\vert#1\right\vert}

\newcommand{\C}{\mathbb{C}}
\newcommand{\R}{\mathbb{R}}

\newcommand{\N}{\mathbb{N}}
\newcommand{\norm}[1]{\left\Vert#1\right\Vert}
\newcommand{\FH}{\mathfrak{H}}
\newcommand{\RE}{\mathbb{Re}}
\newcommand{\IM}{\mathbb{Im}}
\newcommand{\E}{\mathbb{E}}

\newcommand{\innp}[1]{\langle {#1}\rangle}

\date{}
\begin{document}
	
	\title[Berry-Ess\'een bound for complex Wiener-It\^{o} integral]{Berry-Ess\'een bound for complex Wiener-It\^{o} integral}
	%%% \subtitle{Sub-Title of Your Article}%%% optional
	\author{Huiping \textsc{Chen}}% Author Name (\sc should NOT be used here)
	\address{LMAM, School of Mathematical Sciences, Peking University, Beijing 100871, China\\Academy of Mathematics and Systems Science, Chinese Academy of Sciences, Beijing 100190, China}
	\email{chenhp@pku.edu.cn}
	
	\author{Yong \textsc{Chen}}% Author Name (\sc should NOT be used here)
	\address{Corresponding author\\Center of Applied Mathematics, School of Mathematics and Statistics,
		Jiangxi Normal University, Nanchang, Jiangxi 330022, China}
	\email{zhishi@pku.org.cn}
	
	\author{Yong \textsc{Liu}}% Author Name (\sc should NOT be used here)
	\address{LMAM, School of Mathematical Sciences, Peking University, Beijing 100871, China}
	\email{liuyong@math.pku.edu.cn}
	
	\subjclass[2020]{60F05; 60G15; 60H05}% Subject code(s)
	
	\keywords{Berry-Ess\'een bound, complex Wiener-It\^{o} integral, Fourth Moment Theorem. \\\indent 	
	We thank Prof. Xiaohong Lan for valuable comments and discussion. Y. Chen is supported by NSFC (No. 11961033). Y. Liu is supported by NSFC (No. 11731009, No. 12231002) and Center for Statistical Science, PKU.\\\indent 
	This paper has been greatly modified based on \cite{ccl2023}}% Key word(s)

	\begin{abstract}
		For complex multiple Wiener-It\^{o} integral, we present Berry-Ess\'een upper and lower bounds in terms of moments and kernel contractions under the Wasserstein distance. As a corollary, we simplify the previously known contraction condition of the complex Fourth Moment Theorem. Additionally, as an application, we explore the optimal Berry-Ess\'een bound for a statistic associated with the complex-valued Ornstein-Uhlenbeck process.
	\end{abstract}
	
	\maketitle	
	\section{Introduction}
	
	In 1952, It\^o published the groundbreaking article \cite{ito}, establishing the theory of complex multiple Wiener-It\^o integral with respect to a complex normal random measure. Since then, there has been renewed interest in theoretical research on complex multiple Wiener-It\^o integral. For instance, Hida focused on the theory of complex multiple Wiener-It\^o integral within the framework of complex white noise and from an analytical perspective in \cite[Chapter 6]{HD80}. In \cite{Camp15,chw17,cl17}, the authors proved the complex Fourth Moment Theorem, which states that the convergence in distribution of a sequence of complex multiple Wiener-It\^o integrals to a complex Gaussian variable is equivalent to the convergences of its absolute moments up to the fourth. Additionally, the product formula and \"Ust\"unel-Zakai independence criterion for complex multiple Wiener-It\^o integral were obtained in \cite{ch17}. In \cite{ChenLiu2014,cl19}, the authors explored various properties of complex multiple Wiener-It\^o integral, complex Ornstein-Uhlenbeck operator and semigroup. 
	
	The study of complex Gaussian field is highly motivated by various applications. For instance, it is crucial to understand the asymptotic behavior of some functionals of complex  Wiener-It\^o integral in probabilistic model of cosmic microwave background radiation, see \cite{Kamionkowski_1997,marinucci_peccati_2011}. The theory of complex multiple Wiener-It\^o integral can be used to analyze the stochastic complex Ginzburg-Landau equation, which is one of the most significant nonlinear partial differential equations in applied mathematics and physics, see \cite{RevModPhys.74.99,Hoshino2017}. The Ornstein-Uhlenbeck process, introduced in \cite{arato1982linear,arato1962evaluation} to model the Chandler wobble or variation of latitude concerning the rotation of the Earth, has also found widespread use in finance and econophysics. Techniques involving complex multiple Wiener-It\^o integral can be used to obtain statistical inference for some parameter estimators of the Ornstein-Uhlenbeck process, such as consistency and asymptotic normality, see \cite{chw17,sty22} for example. In the field of communication and signal processing, noise is frequently assumed to be complex Gaussian noise, see \cite{Aghaei2008,BARONE2005,Matalkah2008,Reisenfeld2003} for more details.

	In this paper, we investigate the Berry-Ess\'een bound for one-diemensional and multi-dimensional complex multiple Wiener-It\^o integral under the Wasserstein distance. In the one-dimensional case, we derive the Berry-Ess\'een lower and upper bounds in terms of moments in Theorem \ref{cFMBE bound}. Firstly, this upper bound appreciably improves upon that in \cite[Theorem 4.6]{Camp15}, see Remark \ref{compare to Campese 1} and Remark \ref{compare to Campese 2} for specific explanations. Secondly, we express the lower and upper bounds in terms of contractions of kernels, which are more convenient to calculate in practical models. It is worth pointing out that we remove some redundant contractions in the upper bound. As a by-product, we simplify the complex Fourth Moment Theorem presented in \cite[Theorem 1.1]{chw17} and \cite[Theorem 3.14]{cl19}. Thirdly, we aim to derive the optimal Berry-Ess\'een bound under the Wasserstein distance, continuing the work in \cite[Theorem 4.1]{c23}, where the first author of the present paper obtained the optimal bounds for multi-dimensional and complex Wiener-It\^o integrals under some smooth distance. This is undoubtedly very difficult in both the complex and real multi-dimensional cases. At present, we achieve this goal and acquire the optimal Berry-Ess\'een bound in some special cases, see Remark \ref{optimal bound} for details and Theorem \ref{example_BM} for an example. 
	
	In the multi-dimensional case, the Berry-Ess\'{e}en upper bound for a complex multiple Wiener-It\^o integral vector is presented in Theorem \ref{duowei b-e bound thm}. One of the complexities of the multi-dimensional case is that the Berry-Ess\'{e}en upper bound is related to a partial order relation among the indices of its components. This Berry-Ess\'{e}en upper bound may improve the corresponding estimate in the real case (see \cite[Theorem 6.2.2]{NourPecc12}), see Remark \ref{remark3.12} for further explanation. As a corollary, in Theorem \ref{thm Camp}, we get the multi-dimensional Complex Fourth Moment Theorem, which shows that for a sequence of complex multiple Wiener-It\^o integral vectors, componentwise convergence to a complex Gaussian variable is equivalent to joint convergence. In Theorem \ref{thm Pecca TUd}, we establish some sufficient conditions for the asymptotic normality of a sequence of complex square-integrable functionals of a complex isonormal Gaussian process.

	The paper is organized as follows. Section \ref{section 2} introduces some elements of complex multiple Wiener-It\^o integrals. In Section \ref{section 3}, we prove the Berry-Ess\'een upper and lower bounds for complex multiple Wiener-It\^o integral. As an application, in section \ref{Application}, we derive the optimal Berry-Ess\'een bound for a statistic associated with the complex Ornstein-Uhlenbeck process. The Berry-Ess\'een bound for complex multiple Wiener-It\^o integral vector is discussed in Section \ref{section 5}.
	
	\section{Preliminaries}\label{section 2}

	We introduce the mathematical framework of complex isonormal Gaussian process and complex multiple Wiener-It\^o integral. See \cite{cl17, HD80, ito, janson} for more details.
	
	Suppose that $\FH$ is a complex separable Hilbert space with the inner product denoted by $\left\langle \cdot,\cdot\right\rangle_{\mathfrak{H}}$. Recall that a complex isonormal Gaussian process $Z=\set{Z(h):\,h\in \FH}$ defined on a complete probability space $(\Omega, \mathcal{F}, P)$ is a centered symmetric complex Gaussian family in $L^2(\Omega)$ such that
	\begin{align*}
		\mathbb{E}[ Z({h})^2]=0,\quad \mathbb{E}[Z({g})\overline{Z({h})}]=\innp{{g}, {h}}_{\FH},\quad \forall {g},{h}\in \FH.
	\end{align*}  	
	For integers $p,q\geq 0$, let $H_{p,q}(z)$ be the complex Hermite polynomial, also known as the Hermite-Laguerre-It\^o polynomial, given by 
	\begin{equation*}
		\exp\left\{ \lambda\overline{z}+\overline{\lambda}z-2|\lambda|^2\right\} =\sum_{p=0}^{\infty}\sum_{q=0}^{\infty}\frac{\overline{\lambda}^p\lambda^q}{p!q!}H_{p,q}(z),\quad \lambda\in\mathbb{C}.
	\end{equation*}
	For example, $H_{p,0}(z)=z^{p}$ for $p\geq 0$, $H_{0,q}(z)=\overline{z}^{q}$ for $q\geq 0$, $H_{1,1}(z)=\left| z\right|^2-2 $, $H_{1,2}(z)=\overline{z}\left(|z|^2-4\right)$, and $ H_{2,2}(z)=|z|^4-8|z|^2+8$.
Let $\mathscr{H}_{p,q}(Z)$ for $p,q\geq0$ be the closed linear subspace of $L^2(\Omega)$ generated by the random variables 
$$\set{H_{p,q}(Z( {h})): {h}\in \mathfrak{H},\norm{ {h}}_{\mathfrak{H} }=\sqrt2}.$$
	The space $\mathscr{H}_{p,q}(Z)$ is called the $(p, q)$-th Wiener-It\^o chaos of $Z$.

	Let $\Lambda$ be the set of all sequences $\textbf{a}=\left\lbrace a_k \right\rbrace_{k=1}^{\infty}$ of non-negative integers with only finitely many nonzero components and set $|\textbf{a}|=\sum_{k=1}^{\infty}a_k$, $\textbf{a}!=\prod_{k=1}^{\infty}a_k!$. For $p\geq 1$, let $\mathfrak{H}^{\otimes p}$ and $\mathfrak{H}^{\odot p}$ denote the $p$-th tensor product and symmetric tensor product of $\mathfrak{H}$ respectively. Let $ h\tilde{\otimes}g $ denote the symmetrization of $ h\otimes g $ for $h\in \mathfrak{H}^{\otimes p}$ and $g\in  \mathfrak{H}^{\otimes p}$, where $p,q\geq1$. 
	
	Take a complete orthonormal system $\set{e_k}_{k\ge 1}$ in $\mathfrak{H}$. Given two sequences $\mathbf{p}=\set{p_k}_{k=1}^{\infty}, \mathbf{q}=\set{q_k}_{k=1}^{\infty}\in \Lambda$ satisfying $|\mathbf{p}|=p$ and $|\mathbf{q}|=q$, It\^o in \cite[Theorem 13.2]{ito} proved that the linear mapping
	\begin{equation*}
	I_{p,q}\left(\left( \tilde{\otimes}_{k=1}^{\infty}{e}_k^{\otimes p_k}\right) \otimes \left( \tilde{\otimes}_{k=1}^{\infty}\overline{e_k}^{\otimes q_k}\right) \right):= \prod_{k=1}^{\infty}\frac{1}{\sqrt{2^{p_k+q_k}}}H_{p_k,q_k}\left(\sqrt2Z\left({e}_k\right)\right),
\end{equation*}
	provides an isometry from the tensor product $\mathfrak{H}^{\odot p}\otimes\mathfrak{H}^{\odot q}$, equipped with the norm $\sqrt{p!q!}\|\cdot\|_{\mathfrak{H}^{\otimes (p+q)}}$, onto the $(p,q)$-th Wiener-It\^o chaos $\mathscr{H}_{p,q}(Z)$. For any $f\in\mathfrak{H}^{\odot p}\otimes\mathfrak{H}^{\odot q}$, ${I}_{p,q}(f)$ is called the complex $(p,q)$-th Wiener-It\^o integral of $f$ with respect to $Z$. For any $f\in \mathfrak{H}^{\otimes (p+q)}$, we define $$I_{p,q}(f)=I_{p,q}(\tilde{f}),$$ where $\tilde{f}$ is the symmetrization of $f$ in the sense of \cite[Equation 5.1]{ito}. 
	
 Let $\sigma(Z)$ denote the $\sigma$-field generated by $Z$. The complex Wiener-It\^o chaos decomposition of $ L^2(\Omega,\sigma(Z),P)$ implies that $ L^2(\Omega,\sigma(Z),P)$ can be decomposed into an infinite orthogonal sum of the spaces $\mathscr{H}_{p,q}(Z)$. That is, any random variable $F \in L^2(\Omega,\sigma(Z),P)$ admits a unique expansion of the form
	\begin{equation*}
		F=\sum_{p=0}^{\infty}\sum_{q=0}^{\infty} I_{p,q}\left(f_{p,q}\right),	
	\end{equation*}
	where $f_{0,0}=\mathbb{E}[F]$, and $f_{p,q} \in \mathfrak{H}^{\odot p}\otimes\mathfrak{H}^{\odot q}$ for $p+q \geq1$ are uniquely determined by $F$.
	
	 We define $a\wedge b $ as the minimum of $a, b \in \mathbb{R}$. Given $f\in\mathfrak{H}^{\odot a}\otimes\mathfrak{H}^{\odot b}$ and $g\in\mathfrak{H}^{\odot c}\otimes\mathfrak{H}^{\odot d}$, for $i=0,\dots,a\land d$, $j=0,\dots,b\land c$, the $(i,j)$-th contraction of $f$ and $g$ is an element of $\mathfrak{H}^{\otimes (a+c-i-j)}\otimes\mathfrak{H}^{\otimes (b+d-i-j)}$ defined as
	\begin{align*}
		f \otimes_{i, j} g
		&= \sum_{l_{1}, \ldots, l_{i+j}=1}^{\infty}\left\langle f, e_{l_{1}} \otimes \cdots \otimes e_{l_{i}} \otimes \overline{e_{l_{i+1}}} \otimes \cdots \otimes \overline{e_{l_{i+j}}}\right\rangle\otimes\left\langle g, e_{l_{i+1}} \otimes \cdots \otimes e_{l_{i+j}} \otimes \overline{e_{l_{1}}} \otimes \cdots \otimes \overline{e_{l_{i}}}\right\rangle,
	\end{align*}
	and by convention, $f \otimes_{0,0} g=f \otimes g$ denotes the tensor product of $f$ and $g$. The product formula for complex multiple Wiener-It\^o integral, as presented in \cite[Theorem 2.1]{ch17} and  \cite[Theorem A.1]{Hoshino2017}, states that for $f \in \mathfrak{H}^{\odot a} \otimes\mathfrak{H}^{\odot b}$ and $ g \in \mathfrak{H}^{\odot c} \otimes \mathfrak{H}^{\odot d}$, where $a, b, c, d \geq0$,
	\begin{equation}\label{Product_formula}
		I_{a, b}(f) I_{c, d}(g)=\sum_{i=0}^{a \wedge d} \sum_{j=0}^{b \wedge c}\binom{a}{i} \binom{d}{i}\binom{b}{j}\binom{c}{j} i! j! I_{a+c-i-j, b+d-i-j}\left(f \otimes_{i, j} g\right).
	\end{equation}
	It\^o in \cite[Theorem 7]{ito} showed that complex Wiener -It\^o integrals satisfy the isometry property. That is, for $f\in\mathfrak{H}^{\odot a}\otimes\mathfrak{H}^{\odot b}$ and $g\in\mathfrak{H}^{\odot c}\otimes\mathfrak{H}^{\odot d}$,
	\begin{equation}\label{isometry property}
		\E\left[ I_{a,b}(f)\overline{{I}_{c,d}(g)}\right]=\mathbb{1}_{\left\lbrace a=c,b=d \right\rbrace }  a!b!\left\langle f,g\right\rangle _{\FH^{\otimes (a+b)}},
	\end{equation}	
where $\mathbb{1}_C$ denotes the indicator function of a set $C$. For $ f\in   \FH^{\odot p}\otimes \FH^{\odot q} $, according to 
\cite[Equation 2.15]{cl19}, there exists $h\in \FH^{\odot q}\otimes \FH^{\odot p}$ such that 
\begin{equation*}
	\overline{{I}_{p,q}(f)}={I}_{q,p}(h).
\end{equation*}	
We call $h\in \FH^{\odot q}\otimes \FH^{\odot p}$ the reverse complex conjugate of $  f\in   \FH^{\odot p}\otimes \FH^{\odot q} $. For instance, if $\FH = L_{\mathbb{C}}^2(T, \mathcal{B}, \mu)$ consisting of all complex-valued square-integrable functions, where $(T, \mathcal{B})$ is a measurable space and $\mu$ is a $\sigma$-finite measure without atoms, then for $  f\in   \FH^{\odot p}\otimes \FH^{\odot q} $,
	$$  h(t_1,\dots,t_q,s_1,\dots,s_p)=\bar{f}(s_1,\dots,s_p,t_1,\dots,t_q).$$ 
For $f\in \FH^{\odot p}\otimes \FH^{\odot q}$, by \cite[Theorem 3.3]{cl17},
	there exist $u,v\in (\FH_{\R}\oplus \FH_{\R})^{\odot (p+q)}$ such that
	\begin{equation}\label{pp2 Connection}
	{I}_{p,q}(f)=\mathcal{I}_{p+q}(u)+\mi\, \mathcal{I}_{p+q}(v),
	\end{equation}
where $\FH_{\R}$ is the real Hilbert space such that $\FH=\FH_{\R}+\mi \FH_{\R}$ and $\mathcal{I}_{p+q}(\cdot)$ denotes the real  $(p+q)$-th Wiener-It\^{o} integral with respect to some real isonormal Gaussian process over $\FH_{\R}\oplus \FH_{\R}$. Explicit expressions of the kernels $u,v\in (\FH_{\R}\oplus \FH_{\R})^{\odot (p+q)}$ are provided in \cite{ccl22}. 

Fix an integer $d\geq1$. Let $Z=(Z_1,Z_2,\dots, Z_d)'$ be a $d$-dimensional complex-valued random vector, where $Z_j=A_j+\mathrm{i} B_j$ for $1\leq j\leq d$. Throughout the paper, the distribution of $Z$ means the joint distribution of the $2d$-dimensional real-valued random vector $\tilde{Z}=\left(A_{1},\ldots,A_{d}, B_{1},\ldots,B_{d}\right)$. Given two $d$-dimensional complex-valued random vector random vectors $F$ and $G$, the Wasserstein distance between $F$ and $G$ is defined as
\begin{equation*}
	d_{W}(F,G) =\sup_{g\in \mathrm{Lip}(1)} \abs{\E\left[ g(\tilde{F})\right] - \E\left[ g(\tilde{G})\right] }, 
\end{equation*} 
where $\mathrm{Lip}(1)$ is the class of all real-valued functions $g:\mathbb{R}^{2d}\rightarrow \mathbb{R}$ satisfying 
$$\norm{g}_{\mathrm{Lip}}=\sup_{x\neq y, x,y\in\R^{2d}}\frac{|g(x)-g(y)|}{\norm{x-y}_{\R^{2d}}}\leq 1.$$
Here, $\|\cdot\|_{\mathbb{R}^{2d}}$ denotes the Euclidean norm on $\mathbb{R}^{2d}$. For an integer $m\geq1$, we denote by $ \mathcal{N}_m(0,\Gamma)$ an $m$-dimensional real-valued Gaussian vector with a non-negative definite covariance matrix $\Gamma$. Let $\Sigma=(\Sigma(i,j))_{1\leq i,j\leq d}$ be a complex-valued non-negative definite matrix and $\Sigma_{\RE}=(\Sigma_{\RE}(i,j))_{1\leq i,j\leq d}$, $\Sigma_{\IM}=(\Sigma_{\IM}(i,j))_{1\leq i,j\leq d}$ be its real and imaginary parts respectively. We say that $Z=(Z_1,Z_2,\dots, Z_d)'$ satisfying $\E[Z {Z}']=0$ and $\E[Z\overline{Z'}]=\Sigma$, is Gaussian and denote by $Z\sim \mathcal{CN}_d(0,\Sigma)$ if and only if  $	\tilde{Z}\sim \mathcal{N}_{2d}(0,\Sigma')$, where $\Sigma'=\begin{pmatrix}
  	\frac{1}{2}\Sigma_{\RE} & -\frac{1}{2}\Sigma_{\IM} \\
  	\frac{1}{2}\Sigma_{\IM} & \frac{1}{2}\Sigma_{\RE}
  \end{pmatrix}$. We omit the subscript if $d=1$. 
 
 We denote by $c$, $c_1$ and $c_2$ finite positive constants that can vary from line to line. We write $c(a,b)$ to indicate that the constant $c$ depends on some $a,b\in\mathbb{R}$. Let $\overset{d}{\rightarrow}$ denote convergence in distribution and $\overset{a.s.}{\rightarrow}$ denote convergence almost surely. For $z\in\mathbb{C}$, we denote by $\RE z$ and $\IM z$ 
the real and imaginary parts of $z$ respectively.

	\section{Berry-Ess\'{e}en bound for complex Wiener-It\^o integral}\label{section 3}
		\subsection{Main results}\label{section 3.1}

	\begin{theorem}\label{cFMBE bound}% [Fourth Moment Berry-Ess\'{e}en bound]
	Let $p,q\in\mathbb{N}$ be such that $l:=p+q\geq2$ and $F=I_{p,q}(f)$ with $f\in \mathfrak{H}^{\odot p}\otimes \mathfrak{H}^{\odot q} $. Suppose that $\E[\abs{F}^2]= \sigma^2$, $\E[F^2]= a+\mi b $ and $\sigma^2>\sqrt{a^2+b^2}$. Let $N\sim \mathcal{N}_2(0,C)$, where $C= \frac{1}{2}\begin{bmatrix}\sigma^2+a & b\\ b & \sigma^2-a  \end{bmatrix}$. Then
		\begin{align}
			d_{W}(F, N)&\leq 4\sqrt2  \sqrt{ \sum_{r=1}^{l-1} {2r \choose r}}  \frac{ \sqrt{\lambda_1}}{\lambda_2} \sqrt{\E[\abs{F}^4]-2 \big(\E[\abs{F}^2]\big)^2- \big| \E[F^2] \big|^2}\label{upper bound 1}\\
			&\leq c_2(p,q,a,b,\sigma) \sqrt{ \sum_{0<i+j<l}\norm{f {\otimes}_{i,j}h}^2_{\FH^{\otimes(2(l-i-j))}}},\label{upper bound 2}\\
			d_{W}(F, N)&\geq c_1(a,b,\sigma)\max\left\lbrace \left| \E\left[F^3 \right] \right|, \left|\E\left[F^2\overline{F} \right] \right|, \E\left[| F| ^4 \right]-2\left(\E\left[ | F | ^2\right]  \right) ^2-\left| \E\left[ F^2\right] \right| ^2 \right\rbrace \label{lower bound 1}
			\\&\geq c_1(p,q,a,b,\sigma)\max\left\lbrace \mathbb{1}_{\left\lbrace p=q \right\rbrace }\left| \sum_{i=0}^{p} \left\langle f\tilde{\otimes}_{i,p-i}f,h \right\rangle _{\FH^{\otimes (2p)}}\right| ,\right. \label{lower bound 2} \\&\left.\qquad\qquad\qquad \mathbb{1}_{\left\lbrace p=q \right\rbrace }\left| \sum_{i=0}^{p} \left\langle f\tilde{\otimes}_{i,p-i}f,f \right\rangle _{\FH^{\otimes (2p)}}\right| ,\sum_{0<i+j<l}\norm{f {\otimes}_{i,j}h}^2_{\FH^{\otimes(2(l-i-j))}}\right\rbrace ,\nonumber
		\end{align}
		where $h$ is the reverse complex conjugate of $f$ and $$\lambda_1=\frac12[\sigma^2+\sqrt{a^2+b^2}],\,\lambda_2=\frac12[\sigma^2-\sqrt{a^2+b^2}]$$ are the two eigenvalues of the matrix $C$.
	\end{theorem}
		\begin{remark}\label{optimal bound}
		Note that if $p = q$, then the expression $$\sqrt{ \sum_{0<i+j<l}\norm{f {\otimes}_{i,j}h}^2_{\FH^{\otimes(2(l-i-j))}}}$$ provides an optimal bound, meaning that the lower and upper bounds are consistent after neglecting some constant, under the condition that 
		\begin{align*}
			c_1\leq \frac{	\left| \sum_{i=0}^{p} \left\langle f\tilde{\otimes}_{i,p-i}f,h \right\rangle _{\FH^{\otimes (2p)}}\right|+	\left| \sum_{i=0}^{p} \left\langle f\tilde{\otimes}_{i,p-i}f,f \right\rangle _{\FH^{\otimes (2p)}}\right|}{ \sqrt{ \sum_{0<i+j<l}\norm{f {\otimes}_{i,j}h}^2_{\FH^{\otimes(2(l-i-j))}}}}\leq c_2,
		\end{align*}
		where $c_1$ and $c_2$ are finite positive constants.	Refer to Theorem \ref{example_BM} for an example illustrating this case. 
	\end{remark}

	In Theorem \ref{cFMBE bound}, we get the Berry-Ess\'een upper bound \eqref{upper bound 1} and lower bound \eqref{lower bound 1} in terms of moments by combining \cite[Theorem 6.2.2]{NourPecc12} (which provides an upper bound for real multiple Wiener-It\^o integral vector) with \cite[Theorem 4.1]{c23} (which presents an optimal bound for complex multiple Wiener-It\^o integral under some smooth distance). Furthermore, we formulate the lower and upper bounds in terms of contractions of kernels, which are more convenient to calculate in practical models. The proof of Theorem \ref{cFMBE bound} can be found in Section \ref{section 4}. 
	
	By applying Theorem \ref{cFMBE bound}, we derive two interesting corollaries: Corollary \ref{Coro compare to Campese 1} and Corollary \ref{equilent cond}. Our framework in Theorem \ref{cFMBE bound} is more general compared to \cite{Camp15} and does not require the covariance matrix $C$ to be diagonal. Even under the assumption that $C$ is diagonal as in \cite{Camp15}, namely $\E[F^2]=0$, we improve upon the upper bound given in \cite[Theorem 4.6]{Camp15} in following  corollary.  
	
\begin{corollary}\label{Coro compare to Campese 1}
	Let $p,q\in\mathbb{N}$ be such that $l:=p+q\geq2$ and $F=I_{p,q}(f)$ with $f\in \mathfrak{H}^{\odot p}\otimes \mathfrak{H}^{\odot q} $. Suppose that $\E[\abs{F}^2]= \sigma^2$ and $\E[F^2]=0$ (a sufficient condition for  $\E[F^2]=0$ is $p\neq q$). Let $N\sim \mathcal{CN}(0,\,\sigma^2)$. Then
		\begin{equation}\label{Fourth moment BEB2} 
			d_{W}(F, N)\leq    \frac{4 }{\sigma} \sqrt{\sum_{r=1}^{l-1} {2r \choose r}}  \sqrt{\E[\abs{F}^4]-2 \big(\E[\abs{F}^2]\big)^2}.
		\end{equation}
\end{corollary}
		
		\begin{remark}\label{compare to Campese 1}
		This bound appreciably improves the following bound in \cite[Theorem 4.6]{Camp15} as
		\begin{equation}\label{upper bound in campese15}
			d_{W}(F, N)\le \frac{\sqrt2 }{\sigma} \sqrt{ \E[\abs{F}^4]-2 \big(\E[\abs{F}^2]\big)^2 + \sqrt{\frac12\E[\abs{F}^4]\Big(\E[\abs{F}^4]-2 \big(\E[\abs{F}^2]\big)^2 } \Big)}.
		\end{equation} 
		Specifically, for a sequence of complex Wiener-It\^o integrals $\left\lbrace F_n=I_{p,q}(f_n) \right\rbrace _{n\geq1}$, where $f_n\in \mathfrak{H}^{\odot p}\otimes \mathfrak{H}^{\odot q} $, suppose that $\E[\abs{F_n}^2]= \sigma^2$, $\E[F_n^2]=0 $ and $F_n$ converges in distribution to $N\sim \mathcal{CN}(0,\,\sigma^2)$ (which is equivalent to $\E[\abs{F_n}^4]-2 \big(\E[\abs{F_n}^2]\big)^2\rightarrow 0$ as $n\rightarrow \infty$), then by \eqref{upper bound in campese15},
		\begin{equation*}
			d_{W}(F_n, N)\leq c(p,q,\sigma)\Big(\E[\abs{F_n}^4]-2 \big(\E[\abs{F_n}^2]\big)^2  \Big)^{\frac14}.
		\end{equation*}	
		This upper bound is far greater than \eqref{Fourth moment BEB2}, that is 
		\begin{equation*}
			d_{W}(F_n, N)\leq c(p,q,\sigma)\Big(\E[\abs{F_n}^4]-2 \big(\E[\abs{F_n}^2]\big)^2  \Big)^{\frac12}.
		\end{equation*}	 
		We refer readers to Remark \ref{compare to Campese 2} for further explanation.
	\end{remark}

 By \eqref{revised version1}, we know that 
		\begin{equation}\label{upper bound in more contraction}
		\begin{aligned}
			&\E[\abs{F}^4]-2 \big(\E[\abs{F}^2]\big)^2- \big| \E[F^2] \big|^2
			\\\leq &\,c_2(p,q,a,b,\sigma)  \left( \sum_{0<i+j<l}\norm{f {\otimes}_{i,j}h}^2_{\FH^{\otimes(2(l-i-j))}}+ \mathbb{1}_{\left\lbrace p\neq q\right\rbrace }\norm{f\otimes_{p\wedge q,p\wedge q}f}^2_{\FH^{\otimes(2(l-l'))}} \right. \\&  \left.\qquad\qquad\qquad\quad  +\sum_{0<i+j<l^{'}}\norm{f {\otimes}_{i,j}f}^2_{\FH^{\otimes(2(l-i-j))}}
			\right),
		\end{aligned}
		\end{equation}
		where $l'=2(p\wedge q)$. Compared to \eqref{upper bound in more contraction}, we remove redundant contractions $\norm{f {\otimes}_{i,j}f}^2_{\FH^{\otimes(2(l-i-j))}}$ for $0<i+j\leq l'$ and obtain a simplified upper bound in \eqref{upper bound 2}. Combining \eqref{upper bound 2} with the fact that the topology induced by $d_{W}$ is stronger than the topology induced by weak convergence (see \cite[Proposition C.3.1]{NourPecc12}), we obtain the following version of the complex Fourth Moment Theorem.

	\begin{corollary}\label{equilent cond}
Let $p,q \in \mathbb{N}$ be such that $l:=p+q\ge 2$. Consider a sequence of complex multiple Wiener-It\^o integral $\set{F_{n}=I_{p,q}(f_n):=A_n+\mathrm{i}B_n}_{n\geq1}$, where $f_n\in \mathfrak{H}^{\odot p}\otimes \mathfrak{H}^{\odot q} $. If $\E[\abs{F_n}^2]\to \sigma^2$ and $\E[F_n^2]\to a+\mi b $ as $n\to \infty$, then the following statements are equivalent:
	\begin{itemize}
		\item[\textup{(i).}]The sequence $\left\lbrace (A_n,B_n)\right\rbrace _{n\geq1}$ converges in distribution to $N\sim\mathcal{N}_2(0,C)$ as $n\rightarrow \infty$, where $C= \frac{1}{2}\begin{bmatrix}\sigma^2+a & b\\ b & \sigma^2-a  \end{bmatrix}$.
		\item[\textup{(ii).}]  $\norm{f_n{\otimes}_{i,j} h_n}_{\FH^{\otimes ( 2(l-i-j))}}\to 0 $ as $n\rightarrow \infty$ for any $0<i+j< l$, where $h_n\in \mathfrak{H}^{\odot q}\otimes \mathfrak{H}^{\odot p} $ is the reverse complex conjugate of $f_n$.
	\end{itemize}
	\end{corollary} 

	\begin{remark} 
		\begin{itemize}
			\item[\textup{(1).}]
	Theorem \ref{equilent cond} simplifies \cite[Theorem 1.3]{chw17} and \cite[Theorem 3.14]{cl19}, which states that condition (i) is equivalent to that condition (ii) and  $\lim\limits_{n\rightarrow\infty}\norm{f_n\otimes_{i,j} f_n}_{\FH^{\otimes ( 2(l-i-j))}}= 0$ for any $0<i+j< 2(p\wedge q)$. This improvement is useful when applying the complex Fourth Moment Theorem to establish the asymptotic normality of $\set{F_{n}=I_{p,q}(f_n)}_{n\geq1}$. For example, it allows us to significantly shorten the proof of \cite[Lemma 3.6]{chw17} while still ensuring the asymptotic normality of the complex random variable $F_T$ as given in \eqref{F_T}.
		\item[\textup{(2).}] 	There are contraction conditions analogous to  Corollary \ref{equilent cond} in \cite[Proposition 5]{mp08}, although the definition of complex multiple Wiener-It\^o integral in \cite[Section 4]{mp08} differs from that in our framework and is of a complex kernel with respect to a real isonormal process. 
		\end{itemize}
	\end{remark}

		According to \cite[Theorem 1.3]{chw17} or \cite[Theorem 3.14]{cl19}, we know that (i) and (ii) are both equivalent to that $$\E[\abs{F_n}^4]-2 \big(\E[\abs{F_n}^2]\big)^2- \big| \E[F_n^2] \big|^2\rightarrow 0,\quad n\rightarrow \infty.$$
		See \cite[Proposition 3.12]{cl19} or \eqref{revised version1}, \eqref{revised version2} for different expressions of the quantity $\E[\abs{F_n}^4]-2 \big(\E[\abs{F_n}^2]\big)^2- \big| \E[F_n^2] \big|^2$. Compared to the fourth cumulant of the real Wiener-It\^o integral in \cite[Equation 5.2.5, Equation 5.2.6]{NourPecc12}, the expression \eqref{revised version1} is more  intricate and involves more contractions. Due to this reason and the complexity of the symmetrization of $f_n{\otimes}_{i,j} h_n$ (see \cite[Equation 5.1]{ito}), unlike the Fourth Moment Theorem for real Wiener-It\^o integral (see \cite[Theorem 5.2.7]{NourPecc12}), it is uncertain whether conditions (i) and (ii) are equivalent to 
		$$\norm{f_n\tilde{{\otimes}}_{i,j} h_n}_{\FH^{\otimes ( 2(l-i-j))}}\to 0, $$ for any $0<i+j< p+q$, unless the estimate in \eqref{norm inequality} also holds for $f_n\tilde{{\otimes}}_{i,j} f_n$. This illustrates that the theory of complex multiple Wiener-It\^o integral is inherently more complicated than the real one. However, this problem can be ignored since one always calculates $\norm{f_n\otimes_{i,j} h_n}_{\FH^{\otimes ( 2(l-i-j))}}$  rather than more complicated quantity $\norm{f_n\tilde{{\otimes}}_{i,j} h_n}_{\FH^{\otimes ( 2(l-i-j))}}$ to yield the asymptotic normality of $\set{F_{n}=I_{p,q}(f_n)}_{n\geq1}$ by utilizing Theorem \ref{equilent cond}.

	\subsection{Some technical estimates and proof of Theorem \ref{cFMBE bound}}\label{section 4}
	
	In this section, for $k=1,2$, we let $f_k\in \FH^{\odot p_k}\otimes \FH^{\odot q_k}$, $h_k\in \FH^{\odot q_k}\otimes \FH^{\odot p_k}$ be the reverse complex conjugate of $f_k$ and $l_k=p_k+q_k$.
	
	\begin{lemma}\label{basic inequlity}
For $0\leq i\leq p_1\wedge q_2$ and $0\leq j\leq q_1\wedge p_2$,
		\begin{align}
			\norm{f_1\tilde{\otimes}_{i,j} f_2}_{\FH^{\otimes(l_1+l_2-2(i+j))}}&\le \norm{f_1\otimes_{i,j} f_2}_{\FH^{\otimes(l_1+l_2-2(i+j))}}\le \norm{f_1}_{\FH^{\otimes l_1}}\cdot\norm{f_2}_{\FH^{\otimes l_2}},\label{contraction symmetry inequality}\\
			\norm{f_1\otimes_{i,j} f_2}^2_{\FH^{\otimes(l_1+l_2-2(i+j))}}&\le  \frac12\norm{f_1 {\otimes}_{p_1-i,q_1-j} h_1 }^2_{\FH^{\otimes ( 2( i+j))}} +\frac12\norm{f_2 {\otimes}_{p_2-j,q_2-i} h_2 }^2_{\FH^{\otimes ( 2( i+j))}}.\label{norm inequality} \end{align}
	\end{lemma}
	\begin{proof} 
	The first and second inequalities in \eqref{contraction symmetry inequality} are respectively from Minkowski's inequality and the generalized Cauchy-Schwarz inequality (\cite[Lemma 4.1]{Hermine2012}). 
		Combining Fubini's theorem, Cauchy-Schwarz inequality and the fact that 
		\begin{equation}\label{communicative law}
				\norm{f_1 {\otimes}_{i,j} f_2}_{\FH^{\otimes(l_1+l_2-2(i+j))}}=\norm{f_2 {\otimes}_{j,i} f_1}_{\FH^{\otimes(l_1+l_2-2(i+j))}},
		\end{equation}
 we have that 
			\begin{align}
			\norm{f_1\otimes_{i,j} f_2}^2_{\FH^{\otimes(l_1+l_2-2(i+j))}}&=\innp{f_1\otimes_{p_1-i,q_1-j} h_1,\,h_2\otimes_{q_2-i, p_2-j} f_2 }_{\FH^{\otimes(2(i+j))}}\nonumber\\
			&\le \norm{f_1 {\otimes}_{p_1-i,q_1-j} h_1 }_{\FH^{\otimes ( 2( i+j))}} \cdot\norm{f_2 {\otimes}_{p_2-j,q_2-i} h_2 }_{\FH^{\otimes ( 2( i+j))}}\label{before norm inequality}. \end{align}
			Then \eqref{norm inequality} is from the inequality $2ab\le a^2+b^2$ for $a,b\in \mathbb{R}$.
	\end{proof}

		Especially, applying \eqref{norm inequality} to $f_1=f_2=f \in \FH^{\odot p}\otimes \FH^{\odot q}$, we get that for any $0\leq i,j\le p\wedge q$,
	\begin{equation}\label{norm inequality 0}
		2 \norm{f {\otimes}_{i,j} f }^2_{\FH^{\otimes ( 2(l-i-j))}} \le  \norm{f {\otimes}_{p-i,q-j} h }^2_{\FH^{\otimes ( 2( i+j))}} +\norm{f {\otimes}_{p-j,q-i} h }^2_{\FH^{\otimes ( 2( i+j))}},
	\end{equation}
	where $h$ is the reverse complex conjugate of $f$ and $l=p+q$.
	
Suppose that $(p_k,q_k)\in \N^2$, $k=1,2$.  If $(p_1,q_1)\neq (p_2,q_2) $ and $p_1\ge p_2$, $q_1\ge q_2$, then we denote by $$(p_1,\,q_1)\curlyeqsucc (p_2,\,q_2)\mbox{ or }(p_2,\,q_2)\curlyeqprec (p_1,\,q_1).$$  If we remove the condition $(p_1,\,q_1)\neq (p_2,\,q_2) $ in the above definition then `$\curlyeqsucc  $' gives a partial order relation in terms of set theory. It is obvious that 
\begin{align*}
	\mathbb{1}_{\set{ (p_2,q_2)\curlyeqsucc (p_1,q_1)}}&=\mathbb{1}_{\set{p_1< p_2,q_1\le q_2}}+\mathbb{1}_{\set{p_1=p_2,q_1<q_2}} =\mathbb{1}_{\set{p_1\le p_2,q_1< q_2}}+\mathbb{1}_{\set{p_1<p_2,q_1=q_2}},\\
	\mathbb{1}_{\set{ (p_1,q_1)\neq (p_2,q_2)}}&=\mathbb{1}_{\set{(p_2,q_2)\curlyeqsucc (p_1,q_1)}} +\mathbb{1}_{\set{(p_1,q_1)\curlyeqsucc  (p_2,q_2)}} +\mathbb{1}_{\set{p_1>p_2,q_1<q_2}}+\mathbb{1}_{\set{p_1<p_2,q_1>q_2}}.
\end{align*}	
For ease of notations, we write 
	\begin{align*}
		\mathcal{A}(f_1,h_1,f_2,h_2)=&\sum_{0<i+j<l_1}\norm{f_1 {\otimes}_{i,j}h_1}^2_{\FH^{\otimes(2(l_1-i-j))}}+\sum_{0<i+j<l_2}\norm{f_2 {\otimes}_{i,j}h_2}^2_{\FH^{\otimes(2(l_2-i-j))}},\\
		\mathcal{B}(f_1,h_1,f_2,h_2)=&\mathbb{1}_{\set{ (p_2,q_2)\curlyeqsucc (q_1,p_1)}} \norm{f_1}^2_{\FH^{\otimes l_1}}\cdot\norm{f_2\otimes_{p_2-q_1,q_2-p_1}h_2}_{\FH^{\otimes{2l_1} }} \\ &+\mathbb{1}_{\set{ (p_2,q_2)\curlyeqsucc (p_1,q_1)}}\norm{f_1}^2_{\FH^{\otimes l_1}}\cdot\norm{f_2\otimes_{p_2-p_1,q_2-q_1}h_2}_{\FH^{\otimes{2l_1} }} \\
		& +\mathbb{1}_{\set{(p_1,\,q_1)\curlyeqsucc (q_2,\,p_2)}}\norm{f_2 }^2_{\FH^{\otimes l_2}}\cdot\norm{f_1\otimes_{p_1-q_2,q_1-p_2}h_1}_{\FH^{\otimes{2l_2} }} \\  
		&+\mathbb{1}_{\set{(p_1,\,q_1)\curlyeqsucc (p_2,\,q_2)}}\norm{f_2 }^2_{\FH^{\otimes l_2}}\cdot\norm{f_1\otimes_{p_1-p_2,q_1-q_2}h_1}_{\FH^{\otimes{2l_2} }}.
	\end{align*} 
If $f_1=f_2=f$, we write $	\mathcal{A}(f,h):=\mathcal{A}(f,h,f,h)$, where $h$ is the reverse complex conjugate of $f$.	
	\begin{prop}\label{jproo key estimate}
		For $k=1,2$, let $F_k=I_{p_k,q_k}(f_k)$ with $f_k\in \FH^{\odot p_k}\otimes \FH^{\odot q_k}$.
		Then there exists a constant $c$ depending on $p_k,q_k$, $k=1,2$, such that 
		\begin{equation}\label{key estimate}
		\begin{aligned}
			\quad & \mathrm{Cov}(\abs{F_1}^2,\, \abs{F_2}^2)-\abs{\E[F_1\overline{F_2}]}^2-\abs{\E[F_1F_2]}^2
			\le c\left( 	\mathcal{A}(f_1,h_1,f_2,h_2)+	\mathcal{B}(f_1,h_1,f_2,h_2) \right) .
		\end{aligned}
	\end{equation}
	\end{prop}
	\begin{proof}
		 \cite[Lemma 3.1]{ch17} implies that 
		\begin{equation}\label{express 4moment}
		\begin{aligned}
		&\mathrm{Cov}(\abs{F_1}^2,\, \abs{F_2}^2)-\abs{\E[F_1\overline{F_2}]}^2-\abs{\E[F_1F_2]}^2 
	\\=&\, \mathbb{1}_{\left\lbrace l\geq 2\right\rbrace }\sum_{r=1}^{l-1}\psi_r
		 + \mathbb{1}_{\left\lbrace l\geq 1, (p_1,q_1)\neq (p_2,q_2)\right\rbrace } \psi_l +\mathbb{1}_{\left\lbrace l'\geq 2\right\rbrace }\sum_{r=1}^{l'-1}\phi_r+ \mathbb{1}_{\left\lbrace l'\geq1, (p_1,q_1)\neq (q_2,p_2)\right\rbrace }\phi_{l'}\\:= &\,I_1+I_2+I_3+I_4, 
		\end{aligned} 
	\end{equation}
	where  $l= p_1\wedge p_2+ q_1\wedge q_2$, $l'=p_1\wedge q_2+ q_1\wedge p_2$, $m=l_1+l_2$ and 
		\begin{align*}
			\psi_r&=\sum_{i+j=r} {p_1\choose i}{q_1\choose j}{q_2\choose j}{p_2\choose i}  p_1!q_1! p_2!q_2!  \norm{f_1\otimes_{i,j}h_2}^2_{\FH^{\otimes(m-2r)}}, \mbox{ for } 1\leq r\leq l,\\
			\phi_r&=(p_1+p_2-r)!(q_1+q_2-r)!\norm{\sum_{i+j=r} {p_1\choose i}{q_1\choose j}{q_2\choose i }{p_2\choose j}  i!j!\, f_1\tilde{\otimes}_{i,j} f_2}^2_{\FH^{\otimes(m-2r)}}, \mbox{ for } 1\leq r\leq l'.
		\end{align*}
		Applying \eqref{norm inequality} and  \eqref{communicative law}, we have that
		\begin{equation}\label{first term inequality}
		\begin{aligned}
			I_1
			&\le c\sum_{0 <i+j< l} \norm{f_1\otimes_{p_1-i,q_1-j}h_1}^2_{\FH^{\otimes(2(i+j))}}+\norm{f_2\otimes_{p_2-i,q_2-j}h_2}^2_{\FH^{\otimes(2(i+j))}}  \\ &
			\le c	\mathcal{A}(f_1,h_1,f_2,h_2).
		\end{aligned}
	\end{equation}
		By the power means inequality, Minkowski's inequality and a similar argument as in the proof of \eqref{first term inequality}, we have that
		\begin{equation}\label{third term}
		\begin{aligned}
		I_3
			\le  c \sum_{0 <i+j< l'} \norm{f_1\otimes_{i,j}f_2}^2_{\FH^{\otimes(m-2(i+j))}}\le c	\mathcal{A}(f_1,h_1,f_2,h_2).
		\end{aligned}
	\end{equation}
		Applying \eqref{norm inequality}, \eqref{communicative law} and \eqref{before norm inequality}, we have that
		\begin{equation}\label{second term}
		\begin{aligned}
		I_2
			&\le c\Big[ \mathbb{1}_{\set{l\geq1, p_1>p_2,q_1<q_2}}(\norm{f_1\otimes_{p_1-p_2,0}h_1}^2_{\FH^{\otimes 2l}}+\norm{f_2\otimes_{0,q_2-q_1}h_2}^2_{\FH^{\otimes{2l} }})\\
			&\quad \quad+ \mathbb{1}_{\set{l\geq1,p_1<p_2,q_1>q_2}}(\norm{f_1\otimes_{0,q_1-q_2}h_1}^2_{\FH^{\otimes 2l}}+\norm{f_2\otimes_{p_2-p_1,0}h_2}^2_{\FH^{\otimes{2l} }})\\	&\quad\quad +
		\mathbb{1}_{\set{(p_2,\,q_2)\curlyeqsucc (p_1,\,q_1)}}\norm{f_1}^2_{\FH^{\otimes l_1}}\cdot\norm{f_2\otimes_{p_2-p_1,q_2-q_1}h_2}_{\FH^{\otimes{2l_1} }} \\
			&\quad\quad +\mathbb{1}_{\set{(p_1,\,q_1)\curlyeqsucc (p_2,\,q_2)}}\norm{f_2 }^2_{\FH^{\otimes l_2}}\cdot\norm{f_1\otimes_{p_1-p_2,q_1-q_2}h_1}_{\FH^{\otimes{2l_2} }} \Big]\\
			&\le c\left( 	\mathcal{A}(f_1,h_1,f_2,h_2)+	\mathcal{B}(f_1,h_1,f_2,h_2) \right).
		\end{aligned}
	\end{equation}
		Similarly, 
%		\begin{equation}\label{fourth term inequality}
%		\begin{aligned}
%			I_4
%			&\le c\mathbb{1}_{\left\lbrace (p_1,q_1)\neq (q_2,p_2)\right\rbrace }\norm{f_1\otimes_{p_1\wedge q_2,p_2\wedge q_1}f_2}_{\FH^{\otimes{2l_1} }}^2 \\
%			&\le c\Big[ \mathbb{1}_{\set{l'\geq1,p_1>q_2,q_1<p_2}}(\norm{f_1\otimes_{p_1-q_2,0}h_1}^2_{\FH^{\otimes 2l'}}+\norm{f_2\otimes_{ p_2-q_1,0}h_2}^2_{\FH^{\otimes{2l'} }})\\
%		&\quad \quad+ \mathbb{1}_{\set{l'\geq1,p_1<q_2,q_1>p_2}}(\norm{f_1\otimes_{0,q_1-p_2}h_1}^2_{\FH^{\otimes 2l'}}+\norm{f_2\otimes_{0,q_2-p_1}h_2}^2_{\FH^{\otimes{2l'} }})\\	&\quad \quad+
%		 \mathbb{1}_{\set{(q_2,\,p_2)\curlyeqsucc (p_1,\,q_1)}}\norm{f_1}^2_{\FH^{\otimes l_1}}\cdot\norm{f_2\otimes_{p_2-q_1,q_2-p_1}h_2}_{\FH^{\otimes{2l_1} }} \\
%		&\quad \quad+\mathbb{1}_{\set{(p_1,\,q_1)\curlyeqsucc (q_2,\,p_2)}}\norm{f_2 }^2_{\FH^{\otimes l_2}}\cdot\norm{f_1\otimes_{p_1-q_2,q_1-p_2}h_1}_{\FH^{\otimes{2l_2} }} \Big]\\
%			&\le c\left( 	\mathcal{A}(f_1,h_1,f_2,h_2)+	\mathcal{B}(f_1,h_1,f_2,h_2) \right).
%		\end{aligned}
%	\end{equation}
		\begin{equation}\label{fourth term inequality}
			\begin{aligned}
					I_4
					\le c\mathbb{1}_{\left\lbrace (p_1,q_1)\neq (q_2,p_2)\right\rbrace }\norm{f_1\otimes_{p_1\wedge q_2,p_2\wedge q_1}f_2}_{\FH^{\otimes{2l_1} }}^2 
				\le c\left( \mathcal{A}(f_1,h_1,f_2,h_2)+	\mathcal{B}(f_1,h_1,f_2,h_2) \right).
				\end{aligned}
		\end{equation}
		Substituting the inequalities \eqref{first term inequality}-\eqref{fourth term inequality} into the identity \eqref{express 4moment}, we obtain the desired estimate \eqref{key estimate}.
	\end{proof}

	\begin{proposition}\label{be bound prop 63}
		If $F=I_{p,q}(f)$ with $f\in  \FH^{\odot p}\otimes \FH^{\odot q}$ , %such that $\E[F^2]=0$, 
		then we have that
		\begin{equation}\label{contraction moment}
			c_1(p,q)  \mathcal{A}(f,h)\le \E[\abs{F}^4]-2 \big(\E[\abs{F}^2]\big)^2  -\abs{E[{F}^2]}^2\le c_2(p,q)    \mathcal{A}(f,h),
	\end{equation}
		where $h$ is the reverse complex conjugate of $f$.
	\end{proposition}
	%\begin{remark}\blue{It is a key estimate and we need a short proof.} \end{remark}
	\begin{proof}
		Taking $F_1=F_2=F$ in \eqref{express 4moment}, we obtain that 
		\begin{equation}\label{revised version1}
			\quad\E[\abs{F}^4]-2 \big(\E[\abs{F}^2]\big)^2 -\abs{\E[{F}^2]}^2=\mathbb{1}_{\left\lbrace l\geq 2\right\rbrace }\sum_{r=1}^{l-1}\vartheta_r +\mathbb{1}_{\left\lbrace l'\geq 2\right\rbrace }\sum_{r=1}^{l'-1}\varphi_r +	\mathbb{1}_{\left\lbrace l'\geq1, p\neq q\right\rbrace }\varphi_{l'},
	\end{equation}
		where $l=p+q$, $l^{'}=2(p\wedge q)$ and 
		\begin{align*}
		\vartheta_r	&=\sum_{i+j=r}{p\choose i}^2{q\choose j}^2 (p!q!)^2  \norm{f\otimes_{i,j}h}^2_{\FH^{\otimes(2(l-r))}}, \mbox{ for }  1\leq r\leq l,\\
			\varphi_r&=(2p-r)!(2q-r)!\norm{\sum_{i+j=r}\binom{p}{i}\binom{q}{i}\binom{q}{j}\binom{p}{j}i!j!f\tilde{\otimes}_{i,j} f}^2_{\FH^{\otimes(2(l-r))}},\mbox{ for }  1\leq r\leq l'.
		\end{align*}
		Taking $c_1(p,q)=\min\Big\{\binom{p}{i}^2\binom{q}{j}^2 (p!q!)^2: 0<i+j<l\Big\}$ and applying Proposition \ref{jproo key estimate},  we get the conclusion.
	\end{proof}
	\begin{remark}
		Analogue to \eqref{revised version1}, one can show that 
		\begin{equation}\label{revised version2}
				\E[\abs{F}^4]-2 \big(\E[\abs{F}^2]\big)^2 -\abs{\E[{F}^2]}^2=\mathbb{1}_{\left\lbrace l\geq 2\right\rbrace }\sum_{r=1}^{l-1}	\varrho_r+\mathbb{1}_{\left\lbrace l'\geq 2\right\rbrace }\sum_{r=1}^{l'-1} \varsigma_r+	\mathbb{1}_{\left\lbrace l'\geq 1, p\neq q\right\rbrace }\varsigma_{l'},
		\end{equation}
		where $l=p+q$, $l^{'}=2(p\wedge q)$ and
		\begin{align*}
			\varrho_r&=((l-r)!)^2\norm{\sum_{i+j=r}\binom{p}{i}^2\binom{q}{j}^2i!j!f\tilde{\otimes}_{i,j} h}^2_{\FH^{\otimes(2(l-r))}}, \mbox{ for }  1\leq r\leq l,
			\\	\varsigma_r&=\sum_{i+j=r}{p\choose i}{q\choose i}{q\choose j}{p\choose j} (p!q!)^2  \norm{f\otimes_{i,j}f}^2_{\FH^{\otimes(2(l-i-j))}},\mbox{ for }  1\leq r\leq l'.
		\end{align*}
		 Here we both add a new third term in \cite[Equation 3.10, Equation 3.11]{cl19} and revise them as \eqref{revised version1}, \eqref{revised version2}. Note that this minor error of \cite[Equation 3.10, Equation 3.11]{cl19} dose not affect the main results in \cite{cl19}, since by \eqref{norm inequality 0}, if $p\neq q$, 
	$$\norm{f\otimes_{p\wedge q,p\wedge q}f}^2_{\FH^{\otimes(2(l-l'))}}\leq  \norm{f\otimes_{p-p\wedge q,q-p\wedge q}h}^2_{\FH^{\otimes 2l'}}\leq \sum_{0<i+j<l}\norm{f {\otimes}_{i,j}h}^2_{\FH^{\otimes(2(l-i-j))}}.$$
	\end{remark}

	\begin{lem}\label{3moment}
		If $F=I_{p,q}(f)$ with $f\in  \FH^{\odot p}\otimes \FH^{\odot q}$, then we have that
		\begin{equation}\label{3moment_1}
			\E\left[ F^3\right]=\mathbb{1}_{\left\lbrace p=q \right\rbrace }\sum_{i=0}^{p} i!(p-i)!(p!)^2\binom{p}{i}^4\left\langle f\tilde{\otimes}_{i,p-i}f,h \right\rangle _{\FH^{\otimes (2p)}},
		\end{equation}
		where $h$ is the reverse complex conjugate of $f$, and 
		\begin{equation}\label{3moment_2}
			\E\left[ F^2\overline{F}\right]=\mathbb{1}_{\left\lbrace p=q \right\rbrace }\sum_{i=0}^{p} i!(p-i)!(p!)^2\binom{p}{i}^4\left\langle f\tilde{\otimes}_{i,p-i}f,f \right\rangle _{\FH^{\otimes (2p)}}.
		\end{equation}
	\end{lem}
	
	\begin{proof}
		Applying product formula \eqref{Product_formula}, we have that
		\begin{equation*}
			F^2=\sum_{i=0}^{p\wedge q}\sum_{j=0}^{p\wedge q}i!j!\binom{p}{i}\binom{q}{i}\binom{p}{j}\binom{q}{j}I_{2p-(i+j),2q-(i+j)}\left( f\tilde{\otimes}_{i,j}f \right) .
		\end{equation*}
		By isometry property \eqref{isometry property}, we obtain that 
		\begin{align*}
			\E\left[ F^3\right]&=\E\left[ F^2\overline{\overline{F}}\right]=\sum_{i=0}^{p\wedge q}\sum_{j=0}^{p\wedge q}i!j!\binom{p}{i}\binom{q}{i}\binom{p}{j}\binom{q}{j}\E\left[ I_{2p-(i+j),2q-(i+j)}\left( f\tilde{\otimes}_{i,j}f \right) \overline{{I}_{q,p}(h)}\right] \\
			&=\sum_{i=0}^{p\wedge q}\sum_{j=0}^{p\wedge q}i!j!\binom{p}{i}\binom{q}{i}\binom{p}{j}\binom{q}{j}\mathbb{1}_{\left\lbrace 2p-(i+j)=q\right\rbrace }\mathbb{1}_{\left\lbrace 2q-(i+j)=p\right\rbrace } p!q!\left\langle  f\tilde{\otimes}_{i,j}f ,h\right\rangle _{\FH^{\otimes (p+q)}}\\
			&=\mathbb{1}_{\left\lbrace p=q \right\rbrace }\sum_{i=0}^{p} i!(p-i)!(p!)^2\binom{p}{i}^4\left\langle f\tilde{\otimes}_{i,p-i}f,h \right\rangle _{\FH^{\otimes (2p)}}.
		\end{align*}
		Using the similar argument, we can get \eqref{3moment_2}. Then we finish the proof.
	\end{proof}

	\noindent{\it Proof of Theorem~\ref{cFMBE bound}.\,}
	Denote $F=A+\mi B$. Then by \cite[Theorem 3.3]{cl17}, both $A$ and $B$ are real $l$-th Wiener-It\^{o} integrals. It follows from \cite[Theorem 6.2.2]{NourPecc12}) that 
	\begin{align}\label{bds 00}
		d_{W}(F, N)\le  4\sqrt{ \sum_{r=1}^{l-1} {2r \choose r}}  \frac{ \sqrt{\lambda_1}}{\lambda_2} (\abs{x_1}^{\frac12}+\abs{x_2}^{\frac12}),
	\end{align} where $x_1=\E[A^4]-3(\E[A^2])^2,\,x_2=\E[B^4]-3(\E[B^2])^2$. Note that
	\begin{equation}\label{bushengshi 1}
	\begin{aligned}
		x_1+x_2&= \E[A^4]-3(\E[A^2])^2+ \E[B^4]-3(\E[B^2])^2  \\
		&=\E[\abs{F}^4]-2 \big(\E[\abs{F}^2]\big)^2- \big|\E[F^2] \big|^2-2\left(\E\left[A^{2}B^{2} \right]-\E\left[A^{2}\right] \E\left[ B^{2}\right]-2\left(\E\left[ AB\right]  \right) ^2  \right) \\
		&\le \E[\abs{F}^4]-2 \big(\E[\abs{F}^2]\big)^2- \big| \E[F^2] \big|^2,
	\end{aligned}
\end{equation}
	where in the last line we use the fact (see \cite[Lemma 4.8]{cl17})
	\begin{equation*}
		2\left(\E\left[A^{2}B^{2} \right]-\E\left[A^{2}\right] \E\left[ B^{2}\right]-2\left(\E\left[ AB\right]  \right) ^2\right) \geq0,
	\end{equation*}
	which can be proved by combining product formula \eqref{Product_formula}, isometry property \eqref{isometry property} and some combinatorics. Since $x_1,\,x_2\ge 0$ (see \cite[Corollary 5.2.11]{NourPecc12}), we have that 
	\begin{align}\label{bds 01}
		\abs{x_1}^{\frac12}+\abs{x_2}^{\frac12}\le \sqrt{2(x_1+x_2)}.
	\end{align} 
	Substituting \eqref{bushengshi 1} and \eqref{bds 01} into \eqref{bds 00}, we obtain \eqref{upper bound 1}. Combining with Proposition \ref{be bound prop 63}, we get  \eqref{upper bound 2}.
	
		Denote $N=N_1+\mi N_2$. By \cite[Theorem 4.1]{c23}, we further obtain that 
	\begin{align*}
		&d_{W}(F, N)\geq \sup\left\lbrace\left|\E\left[ g(A,B) \right] - \E\left[ g(N_1, N_2) \right]   \right|  \right\rbrace\\
		\geq& c_1(a,b,\sigma)\max\left\lbrace \left| \E\left[F^3 \right] \right|, \left|\E\left[F^2\overline{F} \right] \right|, \E\left[\left| F \right| ^4 \right]-2\left(\E\left[ \left| F \right| ^2\right]  \right) ^2-\left| \E\left[ F^2\right] \right| ^2 \right\rbrace,
	\end{align*}
	where $g:\mathbb{R}^2\rightarrow \mathbb{R}$ runs over the class of all four-times continuously differentiable functions such that $g$ and all of its derivatives of order up to four are bounded by one. Combining Proposition \ref{be bound prop 63} and Lemma \ref{3moment}, we obtain the lower bound \eqref{lower bound 2}.
	%--------------------------------------------------------------------------------------------------------
	{\hfill\large{$\Box$}}\\
	%--------------------------------------------------------------------------------------------------------
	
	%-------------------------------------------------------------------------------------------------------
	\section{Application}\label{Application}

As an application, we consider a complex-valued Ornstein-Uhlenbeck process defined by the stochastic differential equation 
\begin{equation}\label{SDE}
	\mathrm{d}Z_t=-\gamma Z_t\mathrm{d}t+ \mathrm{d}\zeta_t, \quad t\geq0,
\end{equation}
where $Z_0=0$, $\gamma \in\mathbb{C}$ is unknown, and $\zeta_t$ is a complex Brownian motion. That is, $\zeta_t =\left( B_t^1+\mathrm{i} B_t^2\right)/ \sqrt2$, where
$(B_t^1, B_t^2)_{t\geq0}$ is a two-dimensional standard Brownian motion. Suppose that only one trajectory $\left( Z_t\right)_{0\leq t\leq T} $ for $T> 0$ can be observed. Motivated by the work of \cite{HuNualart2010}, Chen, Hu and Wang in \cite{chw17} considered a least squares estimator of $\gamma$ defined as follows by minimizing $\int_{0}^{T}\left|\dot{Z}_{t}+\gamma Z_{t}\right|^{2} \mathrm{d} t$,
\begin{equation*}
	\hat{\gamma}_{T}=-\frac{\int_{0}^{T} \overline{Z_{t}} \mathrm{d} Z_{t}}{\int_{0}^{T}\left|Z_{t}\right|^{2} \mathrm{d} t}=\gamma- \frac{\int_{0}^{T} \overline{Z_{t}} \mathrm{d} \zeta_{t}}{\int_{0}^{T}\left|Z_{t}\right|^{2} \mathrm{d} t}.
\end{equation*}
They proved that $\sqrt{T}\left(\hat{\gamma}_{T}-\gamma\right)$ is asymptotically normal. Namely, as $ T \rightarrow \infty$,
\begin{equation*}
	\sqrt{T}\left[\hat{\gamma}_{T}-\gamma\right]=-\frac{\frac{1}{\sqrt{T}}\int_{0}^{T} \overline{Z_{t}} \mathrm{d} \zeta_{t}}{\frac{1}{T}\int_{0}^{T}\left|Z_{t}\right|^{2} \mathrm{d} t} \overset{d}{\rightarrow} \mathcal{N}_2\left(0, \lambda\mathrm{Id}_2\right),
\end{equation*}
where $\lambda> 0$ is the real part of $\gamma$ and $\mathrm{Id}_2$ denotes the $2\times2$ identity matrix. They showed that the denominator satisfies
\begin{equation*}
	\frac{1}{T}\int_{0}^{T}\left|Z_{t}\right|^{2} \mathrm{d} t\overset{a.s.}{\rightarrow}\frac{1}{2\lambda},
\end{equation*} 
and for the numerator $F_T:=\frac{1}{\sqrt{T}}\int_{0}^{T} \overline{Z_{t}} \mathrm{d} \zeta_{t}=F_{T,1}+\mathrm{i}F_{T,2}$,
\begin{equation*}\label{example}
	\left( F_{T,1},F_{T,2}\right)  \overset{d}{\rightarrow}\mathcal{N}_2\left(0, \frac{1}{4\lambda}\mathrm{Id}_2\right).
\end{equation*} 
Then the asymptotic normality of the estimator $\hat{\gamma}_{T}$ is obtained. One should note that, in \cite{chw17}, the noise considered by Chen, Hu and Wang is a complex fractional Brownian motion with a Hurst parameter belonging to $\left[ 1/2,3/4\right) $. This case involves more complicated calculations and more precise estimations. Here, in order to derive the optimal Berry-Ess\'een bound of $F_T$ (see Theorem \ref{example_BM}), we consider the special case $H=1/2$, namely the case in which the noise is a complex standard Brownian motion.

Define the Hilbert space $\mathfrak{H}=L_{\mathbb{C}}^2\left( \left[ 0,+\infty\right) \right) $ with inner product $\left\langle  f,g\right\rangle_{\mathfrak{H}}= \int_{0}^{\infty}  f(t) \overline{g(t)} \mathrm{d} t$ for any $f,g\in\mathfrak{H}$. Given $f\in\mathfrak{H}^{\odot a}\otimes\mathfrak{H}^{\odot b}$, $g\in\mathfrak{H}^{\odot c}\otimes\mathfrak{H}^{\odot d}$, for $i=0,\dots,a\land d$, $j=0,\dots,b\land c$, the $(i,j)$-th contraction of $f$ and $g$ is the element of $\mathfrak{H}^{\odot (a+c-i-j)}\otimes\mathfrak{H}^{\odot (b+d-i-j)}$ defined by
\begin{align*}
	&f \otimes_{i, j} g\left(t_{1}, \ldots, t_{a+c-i-j} ; s_{1}, \ldots, s_{b+d-i-j}\right) \\
	=&\,\int_{\left[ 0,+\infty\right)^{2 (i+j)}} f\left(t_{1}, \ldots, t_{a-i}, u_{1}, \ldots, u_{i} ; s_{1} \ldots, s_{b-j}, v_{1} ,\ldots, v_{j}\right) \\
	&\qquad \quad g\left(t_{a-i+1}, \ldots, t_{a+c-i-j}, v_{1}, \ldots, v_{j}; s_{b-j+1}, \ldots, s_{b+d-i-j}, u_{1}, \ldots, u_{i}\right) \mathrm{d} \vec{u} \mathrm{d} \vec{v} ,
\end{align*}
where $\vec{u}=\left(u_{1}, \ldots, u_{i}\right)$ and $\vec{v}=\left(v_{1}, \ldots, v_{j}\right)$.

According to \eqref{SDE}, let 
\begin{equation*}
	\psi_{T}(t;s)= \frac{1}{\sqrt{T}} e^{-\overline{\gamma}(t-s)}\mathbb{1}_{\left\lbrace 0\leq s\leq t\leq T\right\rbrace}, \quad 	h_{T}(t;s)=\overline{\psi_{T}(s;t)}=\frac{1}{\sqrt{T}} e^{-\gamma(s-t)}\mathbb{1}_{\left\lbrace 0\leq t\leq s\leq T\right\rbrace},
\end{equation*}
we know that
\begin{align}
	F_T&=\frac{1}{\sqrt{T}}\int_{0}^{T} \overline{Z_{t}} \mathrm{d} \zeta_{t}=\frac{1}{\sqrt{T}}\int_{0}^{T} \int_{0}^{T}e^{-\overline{\gamma}(t-s)}\mathbb{1}_{\left\lbrace 0\leq s\leq t\leq T\right\rbrace} \mathrm{d} \zeta_{t}\mathrm{d}\overline{\zeta_s}=I_{1,1}(\psi_{T}(t;s)),\label{F_T}\\
	\overline{F_T}&=I_{1,1}(h_{T}(t;s)).\nonumber
\end{align}
By the isometry property of complex multiple Wiener-It\^o integral \eqref{isometry property}, we obtain that
\begin{equation*}
	\E\left[ F_T^2\right]  =\left\langle \psi_{T},h_T \right\rangle _{\mathfrak{H}^{\otimes 2}_{\mathbb{C}}}=\int_{0}^{\infty}\int_{0}^{\infty} \psi_{T}\left(t;s \right)  \overline{h_{T}\left( t;s\right) }\mathrm{d}t\mathrm{d}s=0,
\end{equation*}
and as $T\rightarrow\infty$,
\begin{equation*}
	\begin{aligned}
		\E\left[ |F_T|^2\right] & =\left\langle \psi_{T},\psi_{T} \right\rangle _{\mathfrak{H}^{\otimes 2}_{\mathbb{C}}}
		=\frac{1}{T}\int_{0}^{T}\int_{0}^{t}e^{-2\lambda(t-s)}\mathrm{d}s\mathrm{d}t=\frac{1}{2\lambda}+\frac{1}{4\lambda^2T}e^{-2\lambda T}-\frac{1}{4\lambda^2T}
		\rightarrow\frac{1}{2\lambda}.  
	\end{aligned}
\end{equation*}
Since $\lim\limits_{T\rightarrow \infty}\left( 1+\frac{1}{2\lambda T}e^{-2\lambda T}-\frac{1}{2\lambda T}\right)=1 $, for $T$ large enough, $1+\frac{1}{2\lambda T}e^{-2\lambda T}-\frac{1}{2\lambda T}>0$. Let
\begin{equation}\label{normalized F_T}
	F_T^{'}=\left(1+\frac{1}{2\lambda T}e^{-2\lambda T}-\frac{1}{2\lambda T} \right) ^{-\frac{1}{2}}F_T,
\end{equation}
then $\E[ (F_T^{'})^2] =0$ and	$\E[ |F_T^{'}|^2] =1/(2\lambda)$. Now we consider the optimal Berry-Ess\'een bound of $F_T^{'}$ converging in distribution to $Z\sim \mathcal{CN}(0,1/(2\lambda))$ under the Wasserstein distance as $T\rightarrow\infty$.

According to \cite[Lemma 4.6, Lemma 4.7]{c23} and Proposition \ref{be bound prop 63}, we have that 
\begin{align}
	\left| \E\left[F_T^3 \right] \right|&=\left| \left\langle f\tilde{\otimes}_{0,1}f,h \right\rangle _{\FH^{\otimes 2}}+\left\langle f\tilde{\otimes}_{1,0}f,h \right\rangle _{\FH^{\otimes 2}}\right|=0, \label{example_3moment1}\\
	\left|\E\left[F_T^2\overline{F_T} \right] \right| &=\left|  \left\langle f\tilde{\otimes}_{0,1}f,f \right\rangle _{\FH^{\otimes 2}}+\left\langle f\tilde{\otimes}_{1,0}f,f \right\rangle _{\FH^{\otimes 2}}\right| \asymp \frac{1}{\sqrt{T}},\label{example_3moment2}\\
	\E\left[\left| F_T \right| ^4 \right]-&2\left(\E\left[ \left| F_T \right| ^2\right]  \right) ^2-\left| \E\left[ F_T^2\right] \right| ^2\asymp \norm{f {\otimes}_{0,1}h}^2_{\FH^{\otimes2}}+ \norm{f {\otimes}_{1,0}h}^2_{\FH^{\otimes2}}\asymp\frac{1}{T}.\label{example_4moment}
\end{align}
Combining Theorem \ref{cFMBE bound}, \eqref{example_3moment1}, \eqref{example_3moment2} and \eqref{example_4moment}, we obtain the following theorem.
\begin{theorem}\label{example_BM}
	$F_T^{'}$ defined as \eqref{normalized F_T} converges in distribution to $Z\sim \mathcal{CN}(0,1/(2\lambda))$ and there exist two constants $0<c_1<c_{2}<\infty$ only depending on $\lambda$ such that for $T$ large enough,
	\begin{equation*}
		c_1 \frac{1}{\sqrt{T}} \leq d_{W}\left(F_T^{'}, Z\right) \leq c_{2} \frac{1}{\sqrt{T}}.
	\end{equation*}
\end{theorem}

\begin{remark}\label{compare to Campese 2}
As explained in Remark \ref{compare to Campese 1}, applying the upper bound provided by \cite[Theorem 4.6]{Camp15} or \eqref{upper bound in campese15}, we can only get that 
\begin{equation*}
d_{W}\left(F_T^{'}, Z\right) \leq c_{2} \frac{1}{T^{1/4}},
\end{equation*}
which is far greater than the optimal Berry-Ess\'een bound $1/\sqrt{T}$ we obtained as $T\rightarrow\infty$.
\end{remark}

\section{Berry-Ess\'{e}en bound for complex Wiener-It\^o integral vector}\label{section 5}

	%------------------------------------------------------
	
The following theorem demonstrates that the Berry-Ess\'{e}en upper bound for the complex multiple Wiener-It\^o integral vector is related to a partial order relation of the indices of its components. This fact has no real counterpart and indicates the complexities and challenges inherent in exploring the Berry-Ess\'{e}en bound for complex multiple Wiener-It\^o integral vector.

	%------------------------------------------------------
	
	%------------------------------------------------------
	\begin{thm}\label{duowei b-e bound thm}
	 Consider $F=(F_1,\dots,F_d)'$, where $F_j = I_{p_j,q_j} ( f_j)$, $f_j \in \FH^{\odot p_j}\otimes \FH^{\odot q_j} $ and $l_j:=p_j+q_j\geq 2$ for $j = 1, \dots, d$. Assume that $\E[FF']=0$, $\E[F\overline{F'}]=\Sigma$ and $Z\sim \mathcal{CN}(0,\Sigma)$. If $\Sigma$ is positive definite, then 
		\begin{align}\label{Fourth moment BEB2 coro 2}
			d_{W}(F, Z)\le { \frac{2\sqrt{d\lambda_{max}}}{\lambda_{min}}}\sqrt{\E\norm{F}^4-\E\norm{Z}^4},
		\end{align}
		where $\lambda_{max}$ and $\lambda_{min}$ are respectively the maximum and minimum eigenvalue of $\Sigma$ and $\norm{z}=\sqrt{\sum_{i=1}^{d}|z_i|^2}$ for $z=(z_1,\ldots,z_d)\in \C^d$. Moreover, there exists a constant $c$ depending on $(p_j,q_j)$, $1\leq j\leq d$, such that 
		\begin{equation}\label{key estimate 000}
			\begin{aligned}
				\E\norm{F}^4-\E\norm{Z}^4  
				&\le c \sum_{r=1}^{d}\Big\{ \sum_{0<i+i'<l_r}\norm{f_r {\otimes}_{i,i'}h_r}^2_{\FH^{\otimes(2(l_r-i-i'))}}  \\
				& \quad +   \sum_{j\neq r,\,1\le j\le d} \Big[\mathbb{1}_{\set{ (p_j,q_j)\curlyeqsucc (q_r,p_r)}} \norm{f_r}^2_{\FH^{\otimes l_r}}\cdot\norm{f_j\otimes_{p_j-q_r,q_j-p_r}h_j}_{\FH^{\otimes{2l_r} }} \\
				&\quad +\mathbb{1}_{\set{(p_j,q_j)\curlyeqsucc (p_r,q_r)}}\norm{f_r}^2_{\FH^{\otimes l_r}}\cdot\norm{f_j\otimes_{p_j-p_r,q_j-q_r}h_j}_{\FH^{\otimes{2l_r} }} \\
				&\quad +\mathbb{1}_{\set{(p_r,\,q_r)\curlyeqsucc (q_j,\,p_j)}}\norm{f_j }^2_{\FH^{\otimes l_j}}\cdot\norm{f_r\otimes_{p_r-q_j,\,q_r-p_j}h_r}_{\FH^{\otimes{2l_j} }}  \\
				&\quad +\mathbb{1}_{\set{(p_r,\,q_r)\curlyeqsucc (p_j,\,q_j)}}\norm{f_j }^2_{\FH^{\otimes l_j}}\cdot\norm{f_r\otimes_{p_r-p_j,\,q_r-q_j}h_r}_{\FH^{\otimes{2l_j} }}  \Big]\Big\},
			\end{aligned}
		\end{equation}
		where $h_r$ is the reverse complex conjugate of $f_r$ for $1\leq r\leq d$.
	\end{thm}
	\begin{remark}\label{remark3.12}
The estimate \eqref{key estimate 000} has a very special feature. For example, if $d=2$ and $(p_1,q_1)=(5,1),(p_2,q_2)=(3,2)$, then the second term $\sum_{j\neq r}$ in \eqref{key estimate 000} vanishes. On the other hand, if we use the identity \eqref{pp2 Connection} together with \cite[Theorem 6.2.2]{NourPecc12}, some extra terms similar to that in $\sum_{j\neq r}$ appear. Consequently, the estimate \eqref{key estimate 000} may improve the corresponding estimate in \cite[Theorem 6.2.2]{NourPecc12}.
	\end{remark}
	\begin{remark}
		The bound \eqref{Fourth moment BEB2 coro 2} obviously improves the bound given in \cite[Theorem 4.6]{Camp15} as
		\begin{equation*}
			d_{W}(F, Z)\le \frac{\sqrt{2\lambda_{max}}}{\lambda_{min}} \sqrt{ \E\norm{F}^4-\E\norm{Z}^4+ \sum_{j,r=1}^{d}\sqrt{\frac12\E[\abs{F_j}^4]\Big(\E[\abs{F_r}^4]-2 \big(\E[\abs{F_r}^2]\big)^2 } \Big)}.
		\end{equation*} 
		See Remark \ref{compare to Campese 1} and  Remark \ref{compare to Campese 2} for specific explanations of the case $d=1$.
	\end{remark}
	
		\noindent{\it Proof of Theorem \ref{duowei b-e bound thm}.\,}
		We denote by $d \times d$ matrices $\Sigma_{\RE}=(\Sigma_{\RE}(i,j))_{1\leq i,j\leq d}$ and $\Sigma_{\IM}=(\Sigma_{\IM}(i,j))_{1\leq i,j\leq d}$ the real and imaginary parts of $\Sigma=(\Sigma(i,j))_{1\leq i,j\leq d} $ respectively, namely, $\Sigma=\Sigma_{\RE}+\im \Sigma_{\IM}$. Denote $F^{j}=A^{j}+\mi B^{j}$ for $1\leq j\leq d$. Since $\E[F F'] =0$ and $\E[F \overline{F'}] = \Sigma$, the covariance matrix of $	\left(A^{1},\ldots,A^{d}, B^{1},\ldots,B^{d}\right)$ is $\Sigma'=\begin{pmatrix}
			\frac{1}{2}\Sigma_{\RE} & -\frac{1}{2}\Sigma_{\IM} \\
			\frac{1}{2}\Sigma_{\IM} & \frac{1}{2}\Sigma_{\RE}
		\end{pmatrix}$.
	By \cite[Lemma 3.7.12]{GRG13}, the positive definiteness of $\Sigma$ implies that $\Sigma'$ is positive definite. Combining \cite[Theorem 3.3]{cl17} with \cite[Theorem 4.3]{NR12}, we have that
	\begin{align*}
		d_{W}(F, Z)\le { \frac{2\sqrt{d \lambda_{max}}}{\lambda_{min}}}\sqrt{\E\norm{F}^4-\E\norm{Z}^4}.
	\end{align*}
	Since $Z\sim \mathcal{CN}(0,\Sigma)$, we know that 
	$\RE Z_i$ and $\IM Z_i$ are independent and identically distributed as $\mathcal{N}(0,\frac{1}{2}\Sigma(i,i))$ for $1\leq i\leq d$. Moreover,
	\begin{equation*}
		\E[\RE Z_i\RE Z_j] =\E[\IM Z_i\IM Z_j]=\frac{1}{2}\Sigma_{\RE}(i,j),\quad \E[\RE Z_i\IM Z_j]=\frac12\Sigma_{\IM}(i,j),\quad 1\leq i,j\leq d.
	\end{equation*}
	Calculating directly, we have that 
	\begin{equation*}
		\E\norm{Z}^4=\left( \sum_{j=1}^{d}\Sigma(j,j)\right) ^2+\sum_{j,r=1}^{d}\left| \Sigma(j,r)\right| ^2.
	\end{equation*}
	 Hence, 
	\begin{align*}
		\E\norm{F}^4-\E\norm{Z}^4=\sum_{j,r=1}^d \Big\{ \mathrm{Cov}(\abs{F^j}^2,\,\abs{F^r}^2) - \left|\E\left[ F^j\overline{F^r}\right]  \right| ^2\Big\}.
	\end{align*}
	Then the upper bound \eqref{key estimate 000} is an immediate consequence of Proposition \ref{jproo key estimate} and the fact that $\E[F^jF^r]=0$ for $1\leq j,r\leq d$.
	%--------------------------------------------------------------------------------------------------------
	{\hfill\large{$\Box$}}\\

	%--------------------------------------------------------------------------------------------------------
	
	%-------------------------------------------------------------------------------------------------------
	
	%--------------------------------------------------------------------------------------------------------

	%--------------------------------------------------------------------------------------------------------------%
	%-------------------------------------------------------------------------------------------------------

As a consequence of Theorem \ref{duowei b-e bound thm}, we derive the following multivariate complex Fourth Moment Theorem. Compared to \cite[Proposition 4.9]{Camp15} or \cite[Proposition 3.6]{Campese16}, which also imply a similar result, the assumption of the following theorem is more simpler and the proof is completely different from those presented in \cite{Camp15,Campese16}. 
	\begin{theorem}\label{thm Camp}
	Let $\left\lbrace F_n=(F_n^1,F_n^2,\dots, F_n^d)'\right\rbrace _{n\geq1}$ be a sequence of random vectors, where $F_n^j = I_{p_j,q_j} ( f_n^j)$, $f_n^j \in \FH^{\odot p_j}\otimes \FH^{\odot q_j} $  and $l_j:=p_j+q_j\geq 2$ for $j = 1, \dots, d$ and $n\geq1$. Assume that  $\E[F_n F_n^{'}] =0$ and $\E[F_n \overline{F_n^{'}}]\rightarrow\left(\Sigma(i,j) \right) _{1\leq i,j\leq d}$ as $n\to \infty$. Then, the following two conditions are equivalent as $n\to \infty$.
		\begin{itemize}
			\item[\textup{(i)}]
			$F_n$ converges in distribution to $Z\sim \mathcal{CN}_d(0,\Sigma)$.
			\item[\textup{(ii)}] For every $1 \le j  \le d$, $F^j_n$ converges in distribution to $Z_j\sim\mathcal{CN}(0,\Sigma(j,j))$.
		\end{itemize}
	\end{theorem}
	
Using \cite[Section 3.7.6]{GRG13}, the decomposition theorem of complex multiple Wiener-It\^o integral (\cite[Theorem 3.3]{cl17}) and the multivariate normal approximation for real multiple Wiener-It\^o integral (\cite[Theorem 6.2.3]{NourPecc12}), we can prove Theorem \ref{thm Camp}. However, by Corollary \ref{equilent cond} and Theorem \ref{duowei b-e bound thm}, we provide a simpler and more straightforward proof within the framework of complex multiple Wiener-It\^o integral.

\noindent{\it Proof of Theorem \ref{thm Camp}.\,}
	The implication (i)$\Rightarrow$(ii) is trivial, whereas the implication (ii)$\Rightarrow$(i)
		follows directly from Corollary \ref{equilent cond} and Theorem \ref{duowei b-e bound thm}. Specifically, by Corollary \ref{equilent cond},  (ii) implies that for $1\leq j\leq d$, as $n\rightarrow \infty$ ,
		$$\norm{f_n^j{\otimes}_{i,i'} h_n^j}_{\FH^{\otimes ( 2(l_j-i-i'))}}\to 0, \mbox{ for any }0<i+i'<l_j .$$
		Then by \eqref{Fourth moment BEB2 coro 2} and \eqref{key estimate 000}, 
		$$	d_{W}(F_n, Z)\rightarrow 0, \quad n\rightarrow \infty.$$
		Since the topology induced by $d_{W}$ is stronger than the topology induced by weak convergence (see \cite[Proposition C.3.1]{NourPecc12}), we get the conclusion.	
	%--------------------------------------------------------------------------------------------------------
	{\hfill\large{$\Box$}}\\

	The following complex chaotic central limit theorem (see \cite{Hu17} and \cite[Theorem 3]{ HuNua05} for the real version) is a corollary of Theorem \ref{thm Camp}.
	\begin{theorem}\label{thm Pecca TUd}
		Let $\set {F_n}_{n\ge 1}$ be a sequence in $L^2(\Omega)$ such that 
		$F_n$ has the following Wiener-It\^{o} chaos decomposition 
		$$F_n=\sum_{p+q\ge 1}I_{p,q}( f_{n,p,q}).$$ 
		We denote by $G_{n,M}$ the random vector composed by $\left\lbrace I_{p,q}( f_{n,p,q}): 1\leq p+q\leq M \right\rbrace $, where $M\geq 1$. Suppose that  $\E[G_{n,M} G_{n,M}^{'}] =0$ for any $M\geq 1$. Assume that
		\begin{itemize}
			\item[\textup{(i)}] For each $p+q\ge 1$,  $\lim\limits_{n\to \infty} p!q!\norm{f_{n,p,q}}^2_{\FH^{\otimes (p+q)}}= \sigma^2_{p,q}$;
			\item[\textup{(ii)}] $\sigma^2=\sum\limits_{p+q\ge 1} \sigma^2_{p,q}<\infty $;
			\item[\textup{(iii)}] For each $p+q\ge 2$, $\lim\limits_{n\to \infty} \norm{f_{n,p,q} \otimes_{i,j} {h_{n,q,p}}}^2_{\FH^{\otimes (2(p+q-i-j)}}=0$ for any $0\le i\le p$, $0\le j\le q$ and $0<i+j< p+q$, where 
			$  {h_{n,q,p}} $ is the reverse complex conjugate of $f_{n,p,q}$;
			\item[\textup{(iv)}] $\lim\limits_{M\to \infty} \sup\limits_{n\ge 1}\sum\limits_{p+q> M}  p!q!\norm{f_{n,p,q}}^2_{\FH^{\otimes (p+q)}}=0$.
		\end{itemize}
		Then  $ F_n \stackrel{ {d}}{\to} \mathcal{CN}(0,\sigma^2)$ as $n\to \infty$.
	\end{theorem}
	\noindent{\it Proof of Theorem~\ref{thm Pecca TUd} .\,} By the condition that $\E[G_{n,M} G_{n,M}^{'}] =0$ for any $M\geq 1$, we know that $\mathbb{E}[ I_{p,q}\left( f_{n,p,q}\right)^2] =0$ for each $p+q\geq1$. Combining this fact with conditions (i) and (iii), and using Corollary \ref{equilent cond}, we have that for fixed $p,q$, $$I_{p,q}\left( f_{n,p,q}\right)\overset{d}{\rightarrow}N_{p,q}\sim\mathcal{CN}(0,\sigma^2_{p,q}), \quad n\rightarrow \infty . $$
	By Theorem \ref{thm Camp}, we get that for fixed $M\geq1$, $G_{n,M}$ converges in distribution to $N_{M}$, which is a random vector composed by independent complex Gaussian variables $\left\lbrace N_{p,q}: 1\leq p+q\leq M \right\rbrace $. For fixed $M\ge 1$, we set 
	$$	F_n^{(M)}=\sum_{p+q\ge 1}^{M}I_{p,q}( f_{n,p,q}),\quad N^{(M)}=\sum_{p+q\ge 1}^{M} N_{p,q},\quad N=\sum_{p+q\ge 1}N_{p,q}.$$ Then $F_n^{(M)}\overset{d}{\rightarrow}N^{(M)}$. Let $g:\R^2\rightarrow\R$ be such that $\norm{g}_{\mathrm{Lip}}+\norm{g}_{\infty}\leq 1$, where $$\norm{g}_{\mathrm{Lip}}=\sup_{x\neq y, x,y\in\R^2}\frac{|g(x)-g(y)|}{\norm{x-y}_{\R^2}},\quad \norm{g}_{\mathrm{\infty}}=\sup_{x\in\R^2}|g(x)|.$$ Then the Fortet-Mourier distance (also called bounded Wasserstein distance, see \cite[Definition C.2.1]{NourPecc12})  between $F_n$ and $N$ satisfies that
	\begin{align*}
		& \quad d_{\mathrm{FM}}(F_n,\,N)  =\sup_{g} \abs{\E\left[ g(\RE F_n,\IM F_n)\right] - \E\left[ g(\RE N,\IM N)\right] }  \\
		&\le \sup_{g} \Big[ \left| \E\left[ g(\RE F_n,\IM F_n)\right] - \E\left[ g\left( \RE F_n^{(M)},\IM F_n^{(M)}\right) \right]  \right|  +\abs{\E\left[ g\left( \RE F_n^{(M)},\IM F_n^{(M)}\right) \right]  \right. \\&\left. \quad\qquad-\E\left[ g\left( \RE N^{(M)},\IM N^{(M)}\right) \right]  } +\abs{\E\left[ g\left( \RE N,\IM N\right) \right]  -\E\left[ g\left( \RE N^{(M)}, \IM N^{(M)}\right) \right]  }\Big]  \\
		&\le \sqrt{\E\left[ \left| F_n-F_n^{(M)}\right| ^2\right] }+\sqrt{\E\left[ \left| N-N^{(M)}\right| ^2\right] }\\&\quad
		+\sup_{g}\left|\E\left[ g\left( \RE F_n^{(M)},\IM F_n^{(M)}\right) \right] -\E\left[ g\left( \RE N^{(M)},\IM N^{(M)}\right) \right] \right|
		\\
		&\le  \Big(  \sum\limits_{p+q> M}  p!q!\norm{f_{n,p,q}}^2_{\FH^{\otimes (p+q)}} \Big)^{\frac12}+\Big( \sum\limits_{p+q> M} \sigma^2_{p,q}\Big) ^{\frac12}\\&\quad + \sup_{g}\left|\E\left[ g\left( \RE F_n^{(M)},\IM F_n^{(M)}\right) \right] -\E\left[ g\left( \RE N^{(M)},\IM N^{(M)}\right) \right] \right|,
	\end{align*}
	where we use the mean value theorem and Cauchy-Schwarz inequality in the second inequality. Since $F_n^{(M)}\overset{d}{\rightarrow}N^{(M)}$ as $n\rightarrow\infty$ for a fixed $M\geq1$, by \cite[Proposition C.3.4]{NourPecc12}, the third term on the right hand side of the above inequality converges to zero as $n\rightarrow \infty$. Combining conditions (ii) with (iv), first taking the limit as $n\to \infty$ and then as $M\to \infty$, we conclude that $d_{\mathrm{FM}}(F_n,\,N) \to 0 $ as $n\to\infty$. By \cite[Proposition C.3.4]{NourPecc12}) again, we get that $ F_n \stackrel{ {d}}{\to} N\sim\mathcal{CN}(0,\sigma^2)$.
	%--------------------------------------------------------------------------------------------------------
	{\hfill\large{$\Box$}}\\
	%--------------------------------------------------------------------------------------------------------
	
	%-------------------------------------------------------------------------------------------------------
	
	\bibliographystyle{abbrv}	
	\bibliography{refs}

\begin{thebibliography}{10}

\bibitem{Aghaei2008}
A.~S. Aghaei, K.~N. Plataniotis, and S.~Pasupathy.
\newblock Maximum likelihood binary detection in improper complex {G}aussian
  noise.
\newblock In {\em 2008 IEEE International Conference on Acoustics, Speech and
  Signal Processing}, pages 3209--3212, 2008.

\bibitem{RevModPhys.74.99}
I.~S. Aranson and L.~Kramer.
\newblock The world of the complex {G}inzburg-{L}andau equation.
\newblock {\em Reviews of Modern Physics}, 74(1):99--143, 2002.

\bibitem{arato1982linear}
M.~Arat{\'o}.
\newblock {\em Linear stochastic systems with constant coefficients: a
  statistical approach}.
\newblock Lecture Notes in Control and Information Sciences. Springer-Verlag,
  Berlin, 1982.

\bibitem{arato1962evaluation}
M.~Arat{\'o}, A.~N. Kolmogorov, and Y.~G. Sinai.
\newblock Evaluation of the parameters of a complex stationary {G}auss-{M}arkov
  process.
\newblock {\em Doklady Akademii Nauk SSSR}, 146:747--750, 1962.

\bibitem{BARONE2005}
P.~Barone.
\newblock On the distribution of poles of {P}ad\'{e} approximants to the
  {Z}-transform of complex {G}aussian white noise.
\newblock {\em Journal of Approximation Theory}, 132(2):224--240, 2005.

\bibitem{Hermine2012}
H.~Bierm\'{e}, A.~Bonami, I.~Nourdin, and G.~Peccati.
\newblock Optimal {B}erry-{E}sseen rates on the {W}iener space: the barrier of
  third and fourth cumulants.
\newblock {\em ALEA Lat. Am. J. Probab. Math. Stat.}, 9(2):473--500, 2012.

\bibitem{Camp15}
S.~Campese.
\newblock Fourth moment theorems for complex {G}aussian approximation.
  {P}reprint.
\newblock {\em arXiv:1511.00547}, 2015.

\bibitem{Campese16}
S.~Campese, I.~Nourdin, G.~Peccati, and G.~Poly.
\newblock Multivariate {G}aussian approximations on {M}arkov chaoses.
\newblock {\em Electron. Commun. Probab.}, 21:48, 2016.

\bibitem{c23}
H.~Chen.
\newblock Optimal rate of convergence for vector-valued {W}iener-{I}t\^{o}
  integral.
\newblock {\em ALEA Lat. Am. J. Probab. Math. Stat.}, 21(1):179--214, 2024.

\bibitem{ccl2023}
H.~Chen, Y.~Chen, and Y.~Liu.
\newblock An improved complex fourth moment theorem. {P}reprint.
\newblock {\em arXiv:2304.08088}, 2023.

\bibitem{ccl22}
H.~Chen, Y.~Chen, and Y.~Liu.
\newblock Kernel representation formula: {F}rom complex to real
  {W}iener--{I}t\^{o} integrals and vice versa.
\newblock {\em Stochastic Process. Appl.}, 167:104241, 2024.

\bibitem{ch17}
Y.~Chen.
\newblock Product formula and independence for complex multiple
  {W}iener-{I}t{\^o} integrals.
\newblock {\em Advances in Mathematics (China). Shuxue Jinzhan},
  46(6):819--827, 2017.

\bibitem{chw17}
Y.~Chen, Y.~Z. Hu, and Z.~Wang.
\newblock Parameter estimation of complex fractional {O}rnstein-{U}hlenbeck
  processes with fractional noise.
\newblock {\em ALEA. Latin American Journal of Probability and Mathematical
  Statistics}, 14(2):613--629, 2017.

\bibitem{ChenLiu2014}
Y.~Chen and Y.~Liu.
\newblock On the eigenfunctions of the complex {O}rnstein-{U}hlenbeck
  operators.
\newblock {\em Kyoto J. Math.}, 54(3):577--596, 2014.

\bibitem{cl17}
Y.~Chen and Y.~Liu.
\newblock On the fourth moment theorem for complex multiple {W}iener-{I}t{\^o}
  integrals.
\newblock {\em Infinite Dimensional Analysis, Quantum Probability and Related
  Topics}, 20(1):1750005, 2017.

\bibitem{cl19}
Y.~Chen and Y.~Liu.
\newblock Complex {W}iener-{I}t{\^o} chaos decomposition revisited.
\newblock {\em Acta Mathematica Scientia. Series B. English Edition},
  39(3):797--818, 2019.

\bibitem{GRG13}
R.~G. Gallager.
\newblock {\em Stochastic processes: Theory for applications}.
\newblock Cambridge University Press, 2013.

\bibitem{HD80}
T.~Hida.
\newblock {\em Brownian motion}, volume 11 of Applications of Mathematics.
\newblock Springer-Verlag, New York, 1980.

\bibitem{Hoshino2017}
M.~Hoshino, Y.~Inahama, and N.~Naganuma.
\newblock {Stochastic complex Ginzburg-Landau equation with space-time white
  noise}.
\newblock {\em Electronic Journal of Probability}, 22:1--68, 2017.

\bibitem{HuNua05}
Y.~Hu and D.~Nualart.
\newblock {Renormalized self-intersection local time for fractional Brownian
  motion}.
\newblock {\em The Annals of Probability}, 33(3):948--983, 2005.

\bibitem{HuNualart2010}
Y.~Hu and D.~Nualart.
\newblock Parameter estimation for fractional {O}rnstein-{U}hlenbeck processes.
\newblock {\em Statist. Probab. Lett.}, 80(11-12):1030--1038, 2010.

\bibitem{Hu17}
Y.~Z. Hu.
\newblock {\em Analysis on {G}aussian spaces}.
\newblock World Scientific Publishing Co. Pte. Ltd., 2017.

\bibitem{ito}
K.~It{\^o}.
\newblock Complex multiple {W}iener integral.
\newblock {\em Japanese Journal of Mathematics: Transactions and Abstracts},
  22:63--86, 1952.

\bibitem{janson}
S.~Janson.
\newblock {\em {Gaussian Hilbert spaces}}.
\newblock Number 129. Cambridge University Press, 1997.

\bibitem{Kamionkowski_1997}
M.~Kamionkowski, A.~Kosowsky, and A.~Stebbins.
\newblock Statistics of cosmic microwave background polarization.
\newblock {\em Physical Review D}, 55(12):7368--7388, 1997.

\bibitem{mp08}
D.~Marinucci and G.~Peccati.
\newblock High-frequency asymptotics for subordinated stationary fields on an
  {A}belian compact group.
\newblock {\em Stochastic Processes and their Applications}, 118(4):585--613,
  2008.

\bibitem{marinucci_peccati_2011}
D.~Marinucci and G.~Peccati.
\newblock {\em Random fields on the sphere: representation, limit theorems and
  cosmological applications}.
\newblock London Mathematical Society Lecture Note Series. Cambridge University
  Press, 2011.

\bibitem{Matalkah2008}
G.~M. Matalkah, J.~M. Sabatier, and M.~M. Matalgah.
\newblock A generalized likelihood ratio test for detecting targets in
  multiple-band spectral images with improper complex {G}aussian noise.
\newblock In {\em 2008 15th IEEE International Conference on Image Processing},
  pages 1856--1859. IEEE, 2008.

\bibitem{NourPecc12}
I.~Nourdin and G.~Peccati.
\newblock {\em Normal approximations with {M}alliavin calculus: from {S}tein's
  method to universality}.
\newblock Cambridge Tracts in Mathematics. Cambridge University Press, 2012.

\bibitem{NR12}
I.~Nourdin and J.~Rosi\'nski.
\newblock {Asymptotic independence of multiple Wiener-It\^o integrals and the
  resulting limit laws}.
\newblock {\em The Annals of Probability}, 42(2):497--526, 2014.

\bibitem{Reisenfeld2003}
S.~Reisenfeld and E.~Aboutanios.
\newblock A new algorithm for the estimation of the frequency of a complex
  exponential in additive {G}aussian noise.
\newblock {\em IEEE Communications Letters}, 7(11):549--551, 2003.

\bibitem{sty22}
G.~J. Shen, Z.~Tang, and X.~W. Yin.
\newblock Least-squares estimation for the {V}asicek model driven by the
  complex fractional {B}rownian motion.
\newblock {\em Stochastics}, 94(4):537--558, 2022.

\end{thebibliography}

	%-------------------------------------------------------------------------------------------------------
\end{document}